\documentclass[12pt]{amsart}
\usepackage{amsmath,amsfonts,amsthm,amsopn,cite,mathrsfs}
\usepackage{epsfig,verbatim}
\usepackage{subfigure}
\usepackage{hyperref}

\numberwithin{equation}{section}

\newtheorem{theorem}{Theorem}
\newtheorem{def.}{Definition}
\newtheorem{prop.}[theorem]{Proposition}
\newtheorem{lem.}[theorem]{Lemma}
\newtheorem{cor.}[theorem]{Corollary}
\newtheorem{conj.}[theorem]{Conjecture}
\newtheorem{exmp}{Example}

\newtheorem{rem.}[theorem]{Remark}
\newtheorem{result}{Result}

\newcommand{\norm}[2]{
\left\| #2 \right\|_{#1}
}
\newcommand{\ip}[2]{ \langle {#1} ,{#2}  \rangle}
\newcommand{\diag}[1]{
{\mathrm{diag}} \left( #1 \right)
}
\newcommand{\MM}{\mathit{M}}
\newcommand{\HHmp}{{\mathcal H}_{m}^p}
\newcommand{\Mat}[1]{
{M}^{\left( #1 \right)}
}
\newcommand{\Matt}[2]{
{M}^{\left( #1 \right)} \left( #2 \right) 
}
\newcommand\Mate{M}
\newcommand\Ope{O}
\newcommand{\Opp}[2]{
{\Ope}^{\left( #1 \right)}\left( #2 \right) }
\def\mm{\mathit{m}}

\newcommand{\Op}[1]{
{\Ope}^{\left( #1 \right)}}

\newcommand{\kernel}[1]{\mathrm{ker}\left( #1 \right)}
\newcommand{\identity}[1]{\mathsf{id}_{ #1 }}
\newcommand{\range}[1]{\mathsf{ran}\left( #1 \right)} 

\newcommand{\isdef}{\mathrel{\mathrel{\mathop:}=}}

\def\TT{\mathbb{T}}
\def\cF{{\mathcal F}}
\def\Hil{\mathcal{H}}
\def\cA{{\mathcal A}}
\def\cS{{\mathcal S}}
\def\cB{{\mathcal B}}
\def\HH{\mathbb{H}}

\newcommand{\BL}[1]{
{\cB} \left( #1 \right)
}

\def\inv{{^{-1}}}

\def\DPsi{{\widetilde \Psi}}
\def\DPhi{{\widetilde \Phi}}

\newcommand{\Mpm}[0]{
M_m^{p} (\rd)
}

\newcommand{\Lpm}[0]{
{L^{p}_m (\rd)}
}
\def\LtRd{L^2\left(\rd \right)}

\def\zd{\bZ^d}

\newcommand{\hpm}{\Hil^p_{m}}
\newcommand{\lpm}{\ell^p_{m}}

\newcommand{\fatm}{m}
\def\Mmpq{M_{m}^{p,q}}

\newcommand{\field}[1]{\mathbb{#1}}
\newcommand{\bC}{\field{C}}
\newcommand{\bR}{\field{R}}
\newcommand{\bZ}{\field{Z}}
\def\zd{{\bZ^d}}
\def\rdd{{\bR^{2d}}}
\def\rd{{\bR^{d}}}
\def\intrdd{\int_{\rdd}}

\newcommand{\fif}{if and only if}
\newcommand{\psdo}{pseudodifferential operator}
\newcommand{\tfa}{time-frequency analysis}
\newcommand{\stft}{short-time Fourier transform}
\newcommand{\modsp}{modulation space}
\newcommand{\tf}{time-frequency}
\newcommand{\tfs}{time-frequency shift}

\title{The Lifting Property for Frame Multipliers and 
  Toeplitz Operators}
\author{Peter Balazs }
\address{Acoustics Research Institute \\
Austrian Academy of Sciences \\
Dominikanerbastei 16 \\
A-1010 Vienna, Austia}
\email{peter.balazs@oeaw.ac.at} 
\address{Acoustics, Analysis and AI \\
 Interdisciplinary Transformation University Austria \\
Altenberger Stra{\ss}e 66c \\
Science Park 4, OG 2 \\
A-4040 Linz, Austria}
	\author{Karlheinz Gr\"ochenig}
\address{Faculty of Mathematics \\
University of Vienna \\
Oskar-Morgenstern-Platz 1 \\
A-1090 Vienna, Austria}
\email{karlheinz.groechenig@univie.ac.at}
\subjclass[2020]{47A60,42C15}
\date{}
\keywords{Localized frame, spectral invariance, matrix algebra,
  coorbit space, frame multiplier, Toeplitz operator}
\thanks{This research was funded in whole or in part by the Austrian
  Science Fund (FWF) LOFT - Grant-DOI 10.55776/P34624 . For open access purposes, the
  author has applied a CC BY public copyright license to any author
  accepted manuscript version arising from this submission.}

\begin{document}

\begin{abstract}
Frame multipliers are an abstract version of Toeplitz operators in
frame theory and consist of a composition of a multiplication operator
with the analysis and synthesis operators. Whereas the boundedness
properties of frame multipliers on Banach spaces associated to a
frame, so-called coorbit spaces, 
are well understood, their invertibility is much more difficult. We
show that frame multipliers with a positive symbol are Banach space
isomorphisms between the corresponding coorbit spaces. The results
resemble the lifting theorems in the theory of Besov spaces and
modulation spaces. Indeed, the application of the abstract lifting
theorem to Gabor frames yields a new lifting theorem between
modulation spaces. A second application to Fock spaces yields
isomorphisms between weighted Fock spaces.

The main techniques are the theory of localized frames and existence
of inverse-closed matrix algebras. 
\end{abstract} 

\maketitle

\section{Introduction}

A lifting theorem establishes an isomorphism between spaces of
functions, distributions or sequences and their weighted
versions. The basic example is the multiplication operator with a
multiplier sequence $(\mu _k)$, given by 
\begin{equation}
  \label{eq:3}
  c = (c_k)_{k\in K} \mapsto  (\mu _k  c_k )_{k\in K}\, .
\end{equation}
It  maps  the weighted
$\ell^p$-space  $\ell
^p_{\mu }(K)$ with norm $\|c\|_{\ell ^p_\mu} = \Big( \sum
_{k\in K} |c_k|^p \mu _k^p\Big)^{1/p}$ to the unweighted $\ell ^p(K)$ over the index set
$K$. The multiplication operator establishes  an isomorphism
between those spaces and is called
a lifting operator. By means of a lifting
operator one may often reduce the study of
    functions or operators to the unweighted case. 

Although this example is not more than a definition, it becomes
interesting 
when a function space is described by the behavior of its
coefficients with respect to a basis. For  example, the usual Sobolev space on
the torus   is characterized by its Fourier
coefficients as  $\mathbb{H}^s(\TT ) = \{f \in L^2(\TT ): \sum _{k\in \bZ }
|\hat{f}(k)|^2 (1+4\pi ^2|k|^2)^s < \infty \}$  \cite{kat68}.  Let  $\mu _k
= (1+4\pi ^2|k|^2)^{\nu/2}$ be a polynomial weight function, then the operator
$$
M_\mu f (x) =  \sum _{k\in \bZ } \mu _k \hat{f}(k) e^{ i k x} =
(\mathrm{I} - \frac{d^2}{dx^2})^{\nu /2}f(x) 
$$
is a bijection between the Sobolev
spaces $\HH ^s(\TT) $ and $\HH ^{s-\nu }(\TT )$. 

For  the Sobolev space $\mathbb{H}^s(\TT)$ the underlying
basis is the Fourier basis, and a function in $\mathbb{H}^s(\TT )$  is characterized by the size
of its coefficients with respect to this basis. The
lifting operator can be described as a multiplication  operator  on the
coefficients, i.e., as a 
Fourier  multiplier, or directly as a  
a fractional power of the
derivative.  Similar
phenomena can be observed for Besov spaces and \modsp s (see below).

The situation becomes more subtle when a function space is no  longer
described with respect to a basis, but with respect to a frame, which
is a form of \emph{overcomplete} system with stable series
expansions. To illustrate  this substantial difference, assume first that
$\Psi = \{\psi _k: k\in K\} $ is an orthonormal basis. Then the mapping $f\mapsto
  \langle f, \psi_k\rangle _{k\in K}$ is a bijection between the
  spaces $\Hil$ and $\ell^2 (K)$. Thus for bounded symbols $\mu  =
  \left( \mu _k \right)_{k \in K}$
   the sequence $\mu_k \langle f, \psi_k\rangle , k\in K,$ is always in  the
  range of the analysis operator defined by the sequence $\left(
C_{\Psi} f \right)_k = \left< f, \psi_k\right>
$. Consequently, there is always 
  $h\in \Hil $, such that $\mu _k \langle f, \psi_k\rangle = \langle
  h,\psi _k \rangle $. By contrast, if
  $\Psi$ is a redundant frame, then $\mathrm{ran}\, C_{\Psi
    } \neq \ell^2(K) $ and  in general $(\mu _k \langle f, \psi_k\rangle) \not \in \mathrm{ran}\, C_{\Psi
    } $.   In this case  it is still meaningful, but more intricate,
to describe a space by the size of its coefficients with respect to
the given frame $\Psi$. The additional
difficulties introduced by the redundancy can also be seen by the fact that the (synthesis) operator $c \mapsto  \sum c_k \psi _k$
possesses a non-trivial kernel and coefficients in the series expansion are no longer
unique. For the description of associated spaces one therefore uses a
special set of coefficients, namely those with respect to the dual
frame $\{\tilde{\psi _k}: k\in K\}$, so that every $f$ possesses a (non-unique)
series expansion
  $f = \sum _{k\in K} \langle f, \tilde{\psi _k}\rangle \, \psi _k$.
Whereas initially such expansions hold only in a Hilbert space, it is
often  possible to define  associated spaces - $\hpm$ - by imposing a
weighted $\ell ^p$-norm
on these canonical coefficients via
\begin{equation}
  \label{eq:c2}
  \|f\|_{\hpm } = (\big( \sum _{k\in K} |\langle f, \tilde{\psi _k}\rangle |^p m _k^p \Big) ^{1/p}
  \, . 
\end{equation}
It is not completely obvious how and when this is a meaningful
definition. The decisive assumption is that the frame is localized,
and the theory of these spaces is treated in the context  of
localized frames~\cite{forngroech1}. We will provide the necessary
background in Section~\ref{sec:locfram0}.

It is now tempting to follow the example of Sobolev spaces and
consider a map
\begin{equation}
  \label{eq:c3}
  \MM_{\mu, \Psi} f = \sum _{k\in K} \mu _k \, \langle f,
  \tilde{\psi _k}\rangle \, \psi _k \, .
\end{equation}
Once $\hpm $ is well-defined, it is easy to see that 
$\MM_{\mu , \Psi}$ is
bounded between $\Hil ^p _{ \fatm \mu } $  and $\Hil ^p _{\fatm}$, but it
is unclear whether 
$\MM_{\mu , \Psi}$  is indeed an isomorphism as in the case
of a basis. As stated above the difficulty lies in the fact that the
multiplied coefficient sequence $\mu _k \, \langle f, \tilde{\psi
  _k}\rangle $ need not be of the form $\langle h, \tilde{\psi
  _k}\rangle$ and  might even belong to the kernel of the synthesis
operator. An understanding of this phenomenon has been a major
challenge for us, and  this paper is entirely devoted to obtain a
deeper insight into the isomorphism property of $\MM_{\mu , \Psi}$.

Operators of the form~\eqref{eq:c3} are of interest in
their own right and have been investigated both as abstract
\emph{frame multipliers} \cite{xxlmult1} and for special frames, e.g., under the name
Gabor multipliers~\cite{feinow1}. Especially Gabor multipliers
have numerous  applications, e.g., as  quantization operators in physics \cite{aliant1} and for  time-variant filtering
\cite{%doetor09,feinow1,
hlawatgabfilt1} in signal processing. They are
applied in adaptive sound analysis \cite{marcoxxl11}, noise reduction
\cite{Boelcskei01noisereduction}, psychoacoustic modeling
\cite{xxllabmask1},  sound morphing
\cite{olivtor10} and synthesis \cite{DepKronTor07}. 

The  perturbation of invertible frame
multipliers is studied e.g. in ~\cite{balsto09new,uncconv2011}, but it remains difficult to show
that a given frame multiplier is invertible. Our investigation offers
some answers to this question.

Our main theorem asserts that a frame multiplier with a positive
 multiplier sequence 
${\mu} $ is a Banach space isomorphism between the
Banach  spaces associated to the frame. To offer  a flavor of our main
result,  we state a
provocatively inexact version  of such a  result. 
\begin{result} \label{main} {\it
Under suitable conditions on the frame ${\Psi}$ and the  weights
$m$ and $\mu$, the frame multiplier
$\MM_{\mu, \Psi}$  is a Banach space isomorphism between
$\Hil_{m \sqrt{\mu}}^p $ and
$\Hil^p_{m/\sqrt{\mu}}$,  and for some  constants $m,M>0$
one has the norm equivalence 
  $$
m \|f\|_{\Hil_{m \sqrt{\mu}}^p} \leq   \|\MM_{\mu, \Psi}
f\|_{\Hil^p_{m/\sqrt{\mu}}} \leq M \|f\|_{\Hil_{m \sqrt{\mu}}^p} \, . 
  $$ }
\end{result}
The result as stated is, of course, very imprecise. Our  main
contribution is to identify appropriate abstract conditions on the
frame 
${\Psi}$ so that  all symbols in the envisioned theorem are
meaningful and  the results holds rigorously under precisely stated
assumptions. This will be done after some preliminary work in
Section~4, Theorem~\ref{th:multrecip}.    

By and large, 
the proof is motivated by the proof of the lifting theorems for
\modsp s  in \cite{grto10}.  Many analytic arguments there 
can be 
translated into the language of abstract 
frame theory. 
The additional ideas comes from the theory of 
\emph{localized  frames}   and the  \emph{spectral invariance} of
matrix algebras and will be presented in Sections~2 and~3. 

\vspace{3mm}

\textbf{Background.} 
 In functional analysis, lifting theorems are an important feature in
 the theory of function spaces. For instance, the family of Besov
 spaces \cite{triebel83} is 
a three-parameter family $B^{p,q}_s(\rd )$, where $0<p,q < \infty
$ refer to $L^p$-conditions and $s\in \bR $ is a smoothness
parameter. The lifting theorem for Besov spaces states that
$$
D^n B^{p,q}_s(\bR  ) = B^{p,q}_{s-n}(\bR  )    \, ,
  $$
  with  $D=\frac{d}{dx}$ being the usual derivative.
  In simple terms this means that by taking $n$ derivatives, one
  looses $n$ orders of smoothness. However, the lifting theorem says
  more: the derivative $D^n$ is a Banach space isomorphism from  $
  B^{p,q}_s(\bR  )$ onto $ B^{p,q}_{s  - n }(\bR  )$. See~\cite{triebel83} for
  the precise definition  of Besov spaces and their basic theory, in
  particular Thm.~2.3.8. 
  for the lifting theorem.

 A similar lifting theorem occurs in the theory of the original
  version of \modsp s
  $ M^{p,q}_s(\rd  )$ in~\cite{fe06}, where again $s$ is a
  smoothness parameter.  See Section \ref{sec:Gabmult1} for details on
  modulation spaces and their lifting theorems.
  The lifting theorem states again that $D^n
  M^{p,q}_s(\rd  ) = M^{p,q}_{s-n}(\rd  )   $. The  lifting theorems
    for both  Besov spaces and \modsp s involve arguments with Fourier
    multipliers. We will discuss more advanced lifting theorems for
    \modsp s in Section~\ref{sec:Gabmult1}. 

    \vspace{3mm}
    
    \textbf{Toeplitz operators.}
In complex analysis,    the   analogy to frame multipliers is the
class of Toeplitz operators \cite{zh12}. If 
  $f$ is holomorphic, then ${\mu} f$ for some multiplier ${\mu} $ need not
  be holomorphic. To map back into a space of holomorphic functions,
  one uses a suitable projection $P$ (a Bergman or Szeg\"o
  projection) \cite{zh12}, and the resulting operator is then the
  Toeplitz operator $M_\mu =
  P(\mu f)$.
  In frame theory, analyticity corresponds to the properties of the
  frame coefficients $C_{\Psi }f = \langle f, \psi _k\rangle _{k\in
    K}$. Unless $\Psi $ is a Riesz basis, the range $\mathrm{ran}\,
  C_{\Psi }$ is a proper, closed  subspace of $\ell ^2(K)$.  
  After the 
  multiplication with a symbol $\mu $ the resulting sequence $(\mu
  _k \langle f, \psi _k\rangle )_{k\in K}$ is in general no longer in
  that subspace and needs to be projected back to $\mathrm{ran}\,
  C_{\Psi }$. By identifying $\Hil $ with $\mathrm{ran}\,
  C_{\Psi }$ via  the synthesis
  operator $c \mapsto  \sum _{k\in K} c_k \psi _k$, we then obtain
  precisely the 
operator $M_\mu f = \sum _{k\in K} \mu
  _k \langle f, \psi _k \rangle \psi _k$, in other words, a frame
  multiplier. 

To see  that this is not a mere analogy, we will discuss the
explicit example of Fock spaces. The Fock space $\cF ^p_m$ with a
positive weight $m$  consists of all entire functions satisfying
$ \int _{\bC } |F(z)|^p m(z)^p e^{-\pi p
  |z|^2/2} \, dz <\infty $. By using the explicit formula for the
projection onto Fock space, a Toeplitz operator with multiplier
function $\mu $ is defined by
$$
T_\mu F (z)  = \int _{\bC } \mu (w) F(w) e^{\pi \bar{w}z} e^{-\pi
  |w|^2/2} \, dw \, .
$$
Whereas usually $\mu $ is a function, it has been known for a long
time that one may even take rough distributions for the definition of
Toeplitz operators~\cite{CG03}. By taking the Radon measure $\mu = \sum
_{\lambda \in \Lambda } \mu _\lambda \delta _{\lambda}$ with support
in a discrete, relatively separated  set $\Lambda \subseteq \bC $, one
obtains the operator

$$ M_\mu F (z) = \sum \limits_{\lambda \in \Lambda} \mu_\lambda \,
F(\lambda) e^{\pi \bar{\lambda } z} \,  e^{-\pi
  \left|\lambda\right|^2/2}\, . $$
Comparing this to \eqref{eq:c3} one sees that this is a frame multiplier in the strict sense of the definition. As
an example of the main theorem, we will prove a lifting
theorem between Fock spaces. Although this is only a simple example of
the general lifting theorem for abstract frame multipliers, this is a
new result that may be of independent interest in complex analysis.

\begin{result} \it
Under mild assumptions on $\mu $ and   $\Lambda \subseteq \bC $, the
frame multiplier  $M_\mu$ is an isomorphism from
  $\cF^p_{m \sqrt{\mu }}$ to $\cF^p_{\frac{m}{\sqrt{\mu}}}$.
\end{result}
The precise formulation will be given in Section~5.2. \\

    \textbf{Organization of the paper.}
Section \ref{sec:prelnot0}
introduces notation and preliminary facts about frames, frame
multipliers,  and the associated Banach
spaces. Section \ref{sec:locfram0} revisits localized frames and
matrix algebras and offers a rigorous definition of the Banach  spaces
$\hpm $ associated to a frame.    
In Section \ref{sec:mulisom0} we give a precise formulation of our
main result and prove a very general lifting theorem for frame multipliers  between the
associated Banach spaces.  Section \ref{sec:appl0} applies
these findings to Gabor frames and to frames of reproducing kernels in
Fock space.

\section{Preliminaries and Notation} \label{sec:prelnot0}
For a standard reference to functional analysis and operator theory
and the notation used we refer  to \cite{conw1}. In particular we
 denote the space of bounded operators from a normed space $X$ into
a normed  space $Y$  by $\BL{X,Y}$, and set $\BL{X} = \BL{X,X}$.  By the
notation $c \asymp d$, we mean that there are constants $\alpha, \beta >
0$, such that $\alpha c \le d \le \beta c$.

\subsection{Sequence spaces}
We consider a discrete and countable index set $K$. We consider
functions on $K$ as  sequences and denote them as 
$c = \left(c_k\right)_{k \in K}$. We use a short
notation for component-wise operations, e.g. $c \cdot
d \isdef \left( c_k \cdot d_k \right)_{k \in K}$. 
We will represent an operator $M$ between  sequence
spaces over the  index set $K$  by a matrix,  i.e. $M = \left(M_{kl}\right)_{k , l \in K}$. 
The usual $\ell ^p$-space on the index set $K$ is denoted by $\ell
^p(K)$ or simply $\ell ^p$.

Let $m > 0$ be a strictly positive  weight on $K$, i.e., $m_k > 0$ for all $k \in K$. 
We define the space $\ell_{m}^p$ by
$c \in \ell_{m}^{p} \Longleftrightarrow
m \cdot c  \in \ell^p$, with the norm $\| c
\|_{\ell_{m}^p} = \| m \cdot c
\|_{\ell^p}  =  \big(\sum \limits_k m_k^p \left| c_k \right|^p \big)^{1/p} < \infty$. 
For  given $p$  with $1 \le p < \infty$, we denote by $p'$ the
conjugate index satisfying $\frac{1}{p}+\frac{1}{p'} = 1$. 
Then the  dual space of $\lpm $ can be identified with 
$\left( \ell_{m}^p \right)' \cong
\ell_{1/{m}}^{p'}$  with  the duality  
$$\left< {c} , {d} \right>_{\ell_{m}^p \times\ell_{1/{m}}^{p'}} 
= \sum \limits_k c_k \overline{d_k} \, .$$   

\subsubsection{Isomorphism between $\ell^p_{m}$ spaces} \label{sec:isoseq1}
For a sequence $\mu$
we use the notation $\diag{ \mu}$ for the corresponding diagonal
matrix, i.e. $\diag{\mu} = (\delta_{k,l} \, \mu_k)_{(k,l)\in K\times
  K}$. The associated operator is the multiplication operator
$$
\big(\diag{\mu}c\big)_k = \mu _k c_k \qquad k\in K \, .
$$
The primordial lifting theorem, already mentioned in the introduction, asserts that $\diag{\mu}$ is an
isomorphism.

\begin{lem.} \label{lpiso}
If $\mu >0$, then   the operator $\diag{\mu}$ is an isometric isomorphism  between
  $\ell_{\mu m}^p$ and $\ell_{m}^p$ with inverse
  $\diag{\mu}^{-1} = \diag{\mu ^{-1}}$. 
\end{lem.}
This follows from the identity
$$\norm{\ell_{m}^p}{\diag{\mu} c} = \sum
\limits_{k \in K} \left|\mu_k c_k\right|^p m_k^p = \norm{\ell_{\mu
    m}^p} {c} \, .$$
 Consequently,  one has the following isomorphisms
$$ \diag{\mu}: \ell^p_{\mu m} \rightarrow \ell_{m}^p \text{ and } \ell_{m}^p \rightarrow \ell^p_{m/\mu},$$
and 
$$ \diag{\mu}^{-1}: \ell_{m}^p \rightarrow \ell^p_{\mu m} \text{ and } 
\ell^p_{m/\mu} \rightarrow \ell_{m}^p. $$

\subsection{Frames}

A sequence $\Psi = \left( \psi_k\right)_{k\in K}$ in a separable
Hilbert space $\mathcal{H}$ is a  {\sl frame} for
$\mathcal{H}$~\cite{duffschaef1,ole1n}, if there are constants $A,B>0$
such that 
\begin{equation} \label{sec:frambasdef1}
\qquad A \|f\|^2\leq  \sum_{k\in K}|\langle f,\psi_k\rangle|^2 \leq B \|f\|^2
\qquad \text{
  for all } f\in \Hil \, .
\end{equation}
To every sequence $\Psi \subseteq \Hil $ one can associate four
operators: these are the {\em analysis operator} $C_\Psi: \Hil \to \ell ^2(K)$ defined by
$$
\left(C_{\Psi} f\right)_k = \left< f , \psi_k\right>  \qquad k\in K\, ,$$
the   {\em synthesis operator} $D_\Psi : \ell
 ^2(K) \to \Hil$ given by 
$$
D_\Psi c = D_{\Psi} (c_k) = \sum_{k\in K} c_k \psi_k \, ,
$$
the {\em frame operator} $S_\Psi: \Hil \to \Hil$ as 
$$
S_{\Psi} f = \sum_{k\in K} \left< f , \psi_k\right> \psi_k =  D_{\Psi}
C_{\Psi}f \, ,
$$
and the Gram operator $G_\Psi : \ell ^2(K) \to
 \ell ^2(K)$ with the  Gram matrix
$$
 \left({G}_{\Psi}\right)_{k,l} = \left< \psi_l, \psi_k
 \right> \, ,
 $$
 resulting in  $G_{\Psi } = C_{\Psi}
D_{\Psi}$.
These operators are not necessarily well-defined \cite{xxlstoeant11}. The definition of a
 frame \eqref{sec:frambasdef1} guarantees that all four operators are bounded on the
 respective spaces. In fact, $\Psi $ is a frame, if and only if $S_\Psi $ is
 bounded, positive and invertible on $\Hil $ ($C_\Psi$ is bounded, injective and has closed range or $D_\Psi$ is bounded and onto, respectively).
In this case, we can use  the {\em canonical dual frame}
$\widetilde \Psi := (\tilde \psi_k)_{k \in K}$,
given by $\tilde \psi_k = S_{\Psi}^{-1} \psi_k$ for all $k$ to
obtain 
the frame  reconstruction formula
\begin{equation}
  \label{eq:1} 
  f = \sum \limits_{k \in K} \Big< f , \psi_k\Big> \tilde \psi_k =
  \sum \limits_{k \in K} \left< f , \tilde \psi_k\right> \psi_k \text{
    for all  }f \in \Hil \, .
\end{equation}
In short notation, this reads as $D_\Psi C_{\tilde{\Psi}} =
D_{\tilde{\Psi}}C_\Psi =  \identity{\Hil}$. 

Analogously, for two sequences $\Psi$ and $\Phi$ in $\Hil $ one defines the cross-Gram operator by $G_{\Psi,
  \Phi}=C_{\Psi }D_{\Phi}$, which is again bounded, if $\Psi $ and
$\Phi $ are frames. For later use, we record the following identity
from the algebra of Gramian matrices:  we have
\begin{equation}
  \label{eq:re2}
G_{\Psi } G_{\DPsi } =  G_{\Psi , \DPsi } \, 
\end{equation}
Furthermore, using the notation  $B^\dagger$ for the  (Moore-Penrose) pseudo-inverse
of $B$ \cite{olepinv}, the relation between various Gram matrices associated to a
frame $\Psi $ is given by~\cite[Lemma~3.1]{forngroech1}
\begin{equation}
  \label{eq:re7}
  G_{\Psi , \DPsi }= G_{\Psi } ^\dagger G_\Psi \qquad \text{ and }
  \qquad G_{\DPsi }= \big(G_{\Psi } ^\dagger \big)^2 \,  G_\Psi \, .
\end{equation}

\subsection{Frame Multipliers}
The main object for abstract lifting theorems are frame
multipliers \cite{xxlmult1}. The following definition is a bit more general than what
is needed and uses two frames on different Hilbert
spaces. 
\begin{def.}
Let $\mathcal{H}_{1}$ and $\mathcal{H}_{2}$ be Hilbert spaces, let
${\Psi}=\{\psi_{k}\}\subseteq\mathcal{H}_{1}$ and
$\Phi=\{\phi_{k}\}\subseteq\mathcal{H}_{2}$ be two
frames. For given 
a sequence $m =(m_k)_{k\in K} \in \bC^K$,  
the frame multiplier with symbol $m$ is the  operator
$\MM_{m, \Psi, \Phi} : \mathcal{H}_{1} \rightarrow
\mathcal{H}_{2}$  defined by 
\begin{equation} \label{multdef}\MM_{m, \Psi, \Phi}f = \sum_{k \in K} m_k \langle f, \psi_k \rangle \phi_k . \end{equation}
If $m \in \ell^{\infty}(K)$,  then the  multiplier is bounded
 on $\Hil _1$  by $\sqrt{B_\Phi} \sqrt{B_\Psi} \norm{\infty}{m}$,
 where $B_\Phi$ and $B_\Psi$ refer to  the upper frame bounds in
 \eqref{sec:frambasdef1}.  
\end{def.}
If $\Psi = \Phi$, we  write $\MM_{m, \Psi}= \MM_{m, \Psi, \Phi}$. 
If $m$ is unbounded, $\MM_{m, \Psi, \Phi}$ is
possibly unbounded and defined only on a suitable domain. Clearly such
operators can be investigated under other assumptions on
$\Phi$ and $\Psi$~\cite{balsto09new}. 

 In the following we consider multipliers for a frame and its canonical dual, strictly positive
symbols $m>0$, and extend this definition in an natural way
to Banach spaces \cite{rahxxlXX}. 

\section{A Crash Course on Localized Frames}\label{sec:locfram0}

\subsection{Matrix Algebras}

Underlying the notion of localized frames are spaces of  matrices
over the index set $K$ that form an \emph{algebra} that is
\emph{regular} and \emph{inverse-closed}. We impose the following
assumptions on a Banach $^*$-algebra $\cA $ of matrices over $K$,
where $*$ is the usual matrix adjoint: 
 \begin{enumerate}
  \item[(ARI1)] \emph{Regularity.} Associated to $\cA $  is a class of positive weight functions $W $
    on $K$,  such that
    $$
  \cA \subseteq \bigcap \limits_{\begin{array}{c} p \in [1,\infty]
                                 \\ m \in    W\end{array}}
 \BL{\ell^p_{m}} \, .
 $$
We assume that  $W$ contains the constant weight $\mm \equiv 1$, and that, if $\mm
 \in W$, then also $1/\mm \in W$. 
  \item[(ARI2)] \emph{Inverse-closedness.} $\cA $ is inverse-closed in $\cB (\ell ^2(K))$, i.e., if $A\in
    \cA $ and $A$ is invertible on $\ell ^2(K)$, then also $A\inv \in
    \cA $. 
  \end{enumerate}

We  call a matrix algebra  $\cA$ satisfying these conditions an
\emph{ARI-algebra}. A weight  in $W$ is called $\cA$-admissible.

Spectral invariance refers to the phenomenon
that the spectrum of an operator is independent of the space on which
it acts. The prototypical example is a theorem of Wiener on
convolution operators~\cite{wi64}.   In our
context of ARI-algebras, the  spectral invariance is a consequence of
the axioms. 
  \begin{cor.}[Spectral invariance] \label{cor:specinv1}
    Assume that  $A\in \cA $ and $m\in W$. If $A$ is invertible on
    $\ell^2$, then $A$ is also invertible on $\ell_{m}^p$ for all $p \in [1,\infty]$.  
  \end{cor.}
Furthermore, it can be shown that every ARI-algebra is also pseudo-inverse-closed \cite[Thm.~3.4]{forngroech1}.

While the above set of assumptions is a convenient starting point for
the abstract theory of localized frames and lifting theorems, the reader should be aware
that the construction of concrete examples occupies a special place in
Banach algebra theory and in many applications.  
  The construction of inverse-closed matrix algebras and spectral invariance is an important branch of Banach algebra
  theory and analysis, see for instance \cite{Bas90,Krishtal2011,grle06}.

  \begin{exmp}\label{exjaffard}
As a concrete example we briefly discuss the Jaffard algebra of
matrices with off-diagonal decay of polynomial order. In this case the
index set $K$ is any uniformly discrete subset $K$ of $\rd $, i.e.,
$|k-k'|\geq \delta$ for $k\neq k', k,k'\in K$. By mapping a point $k$
to its nearest lattice point in $\delta \zd $, one may assume without
loss of generality that $K=\zd $. The Jaffard algebra $\mathcal{A}_s$ of decay $s>0$
consists of all matrices $A=(a_{kl})_{k,l\in K}$, such that for some
constant $C>0$
\begin{equation}
  \label{eq:c6}
  |a_{kl}| \leq C \, (1+|k-l|)^{-s} \qquad \text{ for all } k,l\in K
  \, .
\end{equation}

If $s> d$, then $\mathcal{A}_s$ is a Banach algebra with respect to
matrix multiplication. Jaffard's Theorem~\cite{jaffard90} asserts that
$\mathcal{A}_s$ is inverse-closed in $\mathcal{B}(\ell ^2(K))$ for
$s>d$. This is  a major result that has inspired a whole line of
research, see e.g. ~\cite{GK10,Koehldorfer2025,sun07a}.

Now let $m$ be a weight that satisfies the condition 
$ m(k +l) \le C  (1+|k|)^t  m(l)$ for all $k,l \in K$ and some $t\geq
0$ and $C>0$. Such a weight
is called $t$-moderate in \tfa . 
 Then it is easy to see that every $A\in \mathcal{A}
_s$ is bounded on $\ell ^p_{m}(K)$ for  $t$ satisfying $t< s-d$ and $1\leq
p \leq \infty $, see, e.g., \cite{forngroech1}. Thus $\mathcal{A} _s$ is an
ARI-algebra with the class of admissible weights containing all 
$t$-moderate weights with $t<s-d$.

Other examples of ARI-algebras are obtained by imposing different
off-diagonal decay on a matrix~\cite{forngroech1}, e.g., $|a_{kl}| \leq C e^{-a
  |k-l|^b}, k,l\in K$, for fixed $a>0$ and $0<b<1$ \cite{Bas90,grle06}. 
\end{exmp}

  \vspace{3mm}
  
  Corollary~\ref{cor:specinv1} ensures that a matrix in $\cA $ that is invertible on
$\ell ^2(K)$ is automatically invertible (and bounded) on
$\lpm $. 
In some situations one
may want to start with the invertibility on some  weighted $\ell
^{2}$-space and deduce the  invertibility on all $\lpm$ with
admissible weight.

The following lemma is an important special case. 
  For its  concise formulation of our next result we use the  notation
  \begin{equation}
    \label{eq:m1}
  B^\mu := \diag{\mu} \cdot B
  \cdot \diag{\tfrac{1}{\mu}}   \,  
  \end{equation}
  for the weighted matrices that are obtained by conjugation with a diagonal
  matrix. 
	
	The following result is elementary, but extremely useful. In
  essence, it reduces the study on the weighted spaces $\ell^2_\mu$ to the unweighted
  space, therefore constituting in some sense the backbone of our lifting theorem. 
\begin{prop.}\label{prop:Greoch0} Let  $\cA $ be an ARI-algebra, $\mu$ be an  admissible
  weight.
	Let $B$ be a matrix such that 
\begin{enumerate} 
\item[(i)] $B^\mu = \diag{\mu} \cdot B
  \cdot \diag{1/\mu} \in \cA$, and  
\item[(ii)] $B$ is invertible in $\ell^{2}_\mu$.
\end{enumerate}
Then $B$ is bounded and invertible on all
$\ell^p_{w}$ for $1\le p\le\infty$ and $w =
\mu \cdot m$ with $m \in W$, in particular on all $\ell^p$.
\end{prop.}
\begin{proof} The boundedness property  of $B:\ell^p_{m \mu} \rightarrow \ell^p_{m \mu} $ is easy and follows from the
 diagram in Figure~\ref{fig:matop1}, as $B^\mu \in \cA$ and $w = m\mu \in W$, 
 \begin{equation}
   \label{eq:c8}
 \ell^p_{m \mu}
  \stackrel{\diag{\mu}}{\longrightarrow} \ell^p_{m}
  \stackrel{B^\mu}{\longrightarrow}
  \ell^p_{m} \stackrel{\diag{1/\mu}}{\longrightarrow}
  \ell^p_{m \mu} \, .   
 \end{equation}
  The diagonal matrices $\diag{\mu}$ and $\diag{1/\mu}$ are
  isomorphisms and $B^\mu$ is bounded on $\ell ^p_{\fatm}$
for all $\fatm \in W$. 
In particular, for $m = \frac{1}{\mu} \in W$, we obtain the
boundedness on $\ell ^p (K)$. 

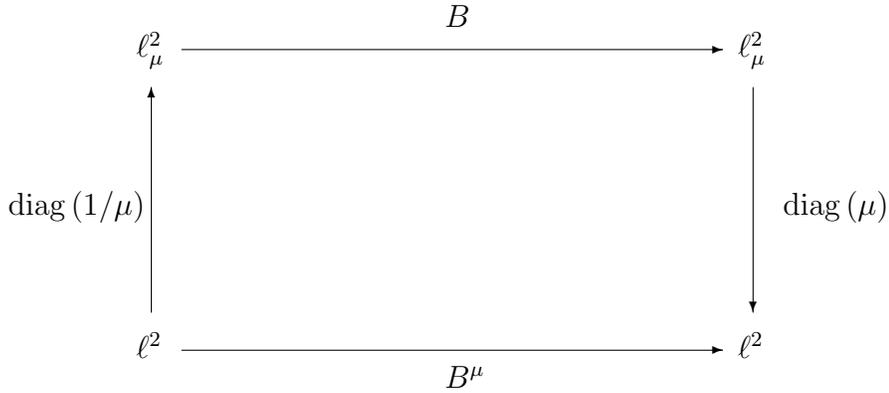
\begin{figure}[ht]
\center
	\begin{picture}(100,60)
	\put(9,51){$\ell^2_\mu$}
	\put(89,51){$\ell^2_\mu$}
	\put(15,52){\vector(4,0){72}}
	\put(50,55){$B$} %O
	\put(11,17){\vector(0,4){30}}
	\put(91,47){\vector(0,-4){30}}
	\put(9,11){$\ell^2$}
	\put(-8,30){$\diag{1/\mu}$}
	\put(95,30){$\diag{\mu}$}
	\put(89,11){$\ell^2$}
	\put(15,12){\vector(4,0){72}}
	\put(50,7){$B^\mu$} 
	\end{picture}
\caption{The relation of $B$ and $B^\mu$} \label{fig:matop1}
\end{figure}

By  assumption (ARI1), 
$B^\mu$ maps $\ell^{2}$ to $\ell^{2}$, because the constant
weight $1$ is admissible. Since  $B$ is assumed
to be invertible on $\ell^{2}_\mu$ and 
 since all three operators  in the
factorization  $ B^{\mu } = \diag{\mu} \cdot
B \cdot \diag{1/\mu}$   are isomorphisms,
$B^\mu$ is then invertible on $\ell^{2}$. 
Because $B^\mu \in \cA$, we also have 
${\left(B^\mu\right)}^{-1} \in \cA$.

Consider the expected inverse of $B$ of the form $C 
= \diag{1/\mu} \cdot {\left(B^\mu\right)}^{-1} \cdot
\diag{\mu}$. Then, as in \eqref{eq:c8} 
$C$ is bounded on all $\ell^p_{m \mu}$ with $\fatm
\in W$. We compute the product  
$$C \cdot B 
= \diag{1/\mu}\cdot \Big({B^\mu}\Big)^{-1}\cdot \underbrace{\diag{\mu} \cdot B \cdot \diag{1/\mu}}_{B^\mu} \cdot \diag{\mu} = 
\identity{\ell^p_{m \mu}}.$$
Likewise 
\begin{eqnarray*} B \cdot C & = & \diag{1/\mu} \cdot \diag{\mu} \cdot B \cdot \diag{1/\mu} \cdot {\left(B^\mu\right)}^{-1} \cdot \diag{\mu} = \\
& = &
\diag{1/\mu} B^\mu {\left(B^\mu\right)}^{-1} \diag{\mu} = \identity{\ell^p_{m \mu}}.
\end{eqnarray*}
Consequently, $C$ is the inverse of $B$ 
on $\ell^p_{m \mu}$, so $B$ is invertible on
all $\ell^p_{m \mu}$.
\end{proof}

\subsection{$\cA$-Localized Frames  and Coorbit Spaces}
The idea of coorbit spaces is to define  Banach spaces attached to a given frame by
imposing conditions on the (dual) frame coefficients. Informally, 
one can impose a norm
$$
\|f\|_{\hpm} = \Big( \sum _{k\in K} |\langle f, \tilde{\psi_k }\rangle | m_k^p
\Big)^{1/p} = \|C_{\Psi} f \|_{\lpm } \,.
$$
As long as $\lpm \subseteq \ell ^2$, the frame expansion \eqref{eq:1}
makes sense, and then $\HHmp $, defined as those elements where this norm is
finite, is a subspace of $\Hil$. If $\lpm \not \subseteq \ell ^2 $, one
takes an abstract norm completion. 

For a meaningful theory of $\hpm $ one needs the concept of localized
frames~\cite{bacahela06,forngroech1,gr04-1}. These are 
central in all contexts  where frame theory is used to describe
properties of functions and distributions beyond mere Hilbert
space properties \cite{gr04-1,Dahlkeetal10, Balan2011}.
In this section, we review the concept of self-localized frames from \cite{forngroech1}, see also \cite{xxlgro14} for details.

\begin{def.} \label{def:localfram0}
Let $\cA$ be an ARI-algebra. Two frames   $\Psi$ and $\Phi$
in a Hilbert space $\Hil$ are said to be  mutually $\cA
$-localized, if their cross-Gramian matrix  $G_{\Psi,
  \Phi}$ 
	is in
$\cA$. We write  $\Psi \sim_{\cA} 
\Phi$. 

 If $\Psi \sim_{\cA} \Psi$, i.e., $G_{\Psi } \in \cA $,
 then 
 $\Psi $ is called an $\cA$-localized  frame. 
\end{def.}  
We mention that $\sim _{\cA}$ is an equivalence relation on the
collection of all $\cA$-localized frames. 
Let $\Psi $ be a frame with canonical dual frame $\DPsi $. 
Let $\Hil^{0 0}= \mathrm{span} \, \{\psi _k: k\in K\}$ be the subspace
of finite linear combinations of vectors in $\Psi $.

\begin{def.}
The Banach space  $\HHmp(\Psi,\DPsi)$  is defined as  the completion of
$\Hil^{0 0} $  with respect to the norm
$$
\norm{\HHmp}{f} =
\norm{\ell^p_{m}}{C_{\DPsi}(f)}$$ for $1 \le p <
\infty $.

For $p = \infty$ we  consider two  completions. First
$\Hil_{m}^0$ is completion  of $\Hil _{00} $ with
respect to the  norm $
\norm{\ell^\infty_{m}}{C_{\DPsi}(f)}$, and second  
$\Hil_{m}^{\infty}$ is   the completion of $\Hil^{0 0} $ in the weak$^*$-topology
with respect to duality  $\sigma(\Hil,
\Hil^{0 0} )$, such that   $\norm{\Hil_{m}^\infty}{f} =
\norm{\ell^\infty_{m}}{C_{\tilde \psi}(f)}$ is finite. 
\end{def.}
  Alternatively,
$\Hil_{m}^{\infty}$ can be identified with the bidual of
$\Hil _{m}^0$. For the subtleties of
this completion we refer to  \cite{xxlgro14}. 

Since $\hpm $ is defined by a property of the transform $f \mapsto
C_\Psi f $, we call $\hpm $ a \emph{coorbit} space with respect the
the frame $\Psi $, in analogy to the theory of coorbit spaces with
respect to group representations where the transform is a
representation coefficient~\cite{fg89jfa}.

The main theorem about localized frames ~\cite{forngroech1} is the fact that their
canonical  dual frames are also localized.

\begin{theorem}\label{locmain}
  If $\Psi $ is $\cA $-localized, then so is its canonical dual frame
  $\tilde{\Psi}$. The frames $\Psi $ and $\tilde{\Psi}$ are mutually $\cA
  $-localized, i.e., $\Psi \sim_{\cA} \Psi$ $\Longrightarrow$ $\DPsi \sim_{\cA} \DPsi$ and $\DPsi \sim_{\cA} \Psi$.
\end{theorem}
In the language of the Gram matrices this fact is expressed by the
implication
\begin{equation}
  \label{eq:c7}
  G_{\Psi, \Psi } \in \cA \quad \Rightarrow \quad  G_{\tilde{\Psi },
    \tilde{\Psi}  } \in \cA   \quad \text{ and } \,\,  G_{
    \tilde{\Psi},\Psi  } \in \cA  \, .
\end{equation}

It follows  that the synthesis operator
 can be extended to $\lpm $: $D_{\Psi}$ is  bounded  from
 $\ell_{m}^p $ to $ \HHmp$
for all $1 \le p  \leq \infty$ and $p = 0$. This   follows from
Theorem~\ref{locmain}, because for every $\cA $-admissible weight $m$
\begin{equation}
  \label{eq:dbound}
\|D_\Psi c \|_{\hpm } =
\|C_{\tilde{\Psi}} D_\Psi c\|_{\lpm }= \|G_{\tilde{\Psi },\Psi}
  c\|_{\lpm} \leq   C \, \|c\|_{\lpm } \, .  
\end{equation}
Likewise the analysis operator $C_{\Psi } $ can be extended to a
bounded operator from $\hpm $ to $\lpm $ for every $\cA$-admissible
weight $m$. As a consequence \cite{forngroech1,xxlgro14}  the
frame reconstruction formula 
\begin{equation} \label{eq:locreconstr1}
f = \sum \limits_k \Big< f ,  \psi_k\Big> \tilde \psi_k  = \sum \limits_k \left< f , \tilde  \psi_k\right> \psi_k ,
\end{equation}
holds on $\hpm $  for $1 \le p \leq \infty$ and $\cA $-admissible
weight $m$. The convergence is  unconditional  for $1 \le p < \infty$ and weak-* unconditional convergence for $p = \infty$.

Theorem~\ref{locmain} is based strongly on the assumption that the
Gram matrix of $\Psi $ is in an inverse-closed matrix algebra
$\cA$. We consider it our main contribution to have identified those
appropriate abstract conditions on frames  (localization and inverse-closedness)
that  raise the   theory of lifting theorems to this highly abstract level.

Another important consequence of Theorem~\ref{locmain} is the
independence of the coorbit space $\hpm $ of the particular frame.
\begin{prop.}\label{locmain2}
Let $\Phi $ and $\Psi $ be two frames for $\Hil $. Assume that
$\Psi $ is $\cA $-localized and $ \Phi \sim_{\cA} \Psi$. Then the
spaces $\Hil_{m}^p ( \Psi, \DPsi)$, $\Hil_{m}^p (
\DPsi, \Psi)$,   and $ \Hil_{m}^p ( \Phi, \DPhi)$ coincide
and the respective  norms are equivalent, i.e. 
\begin{equation}\label{eq:eqnorm1}
\norm{\HHmp}{f} = \norm{\ell_{m}^p}{C_{\DPsi}f}  \asymp
\norm{\ell_{m}^p}{C_{\Psi} f} \asymp  \norm{\ell_{m}^p}{C_{\Phi}f} \, .
\end{equation}
\end{prop.}
It therefore makes sense to write unambiguously $\HHmp :=
\Hil_{m}^p ( \Psi, \DPsi) =  \Hil_{m}^p (\DPsi,
\Psi)$.  Thus $\hpm $ depends only on the $\sim _{\cA }
$-equivalence class of a  frame. 

The duality of the coorbit spaces is inherited from the $\ell
^p$-spaces:  we have $\left(\Hil_{m}^p\right)' =
\Hil_{1/m}^{p'}$ for $1 \le p < \infty$ with the duality
relation  
\begin{equation} \label{eq:dualrel}
\left< f , g \right>_{
\Hil_{m}^p \times \Hil_{1/m}^{p'}
} = 
\left< C_{\DPsi} f, C_{\Psi} g  \right>_{\ell^p_m \times
  \ell^{p'}_{1/m}} = \sum _{k\in K } \langle f, \tilde{\psi
  _k} \rangle \, \overline{\langle g, \psi _k\rangle} \, .
\end{equation}

By this construction also other properties known for the Hilbert space can be transferred to the coorbit spaces, like the fact that $D_\Psi^* = C_\Psi$, and related to that the following result \cite{xxlgro14}:
\begin{lem.}\label{lem:gramoper1} 
Let  $\Psi$ be a $\cA$-localized frame. 
Then the Gram matrix 
$G_{\Psi,\DPsi}$ is the orthogonal  projection from
$\ell ^2$  onto the range of $C_\Psi$ with kernel
$\kernel{D_{\Psi}}$. The extension of $G_{\Psi,\DPsi}$ is a
bounded projection from $\ell_{m}^p$  onto the range of
$C_\Psi \subseteq \lpm $ with kernel
$\kernel{D_{\Psi}}\subseteq \lpm $. 
In addition, $ \ell_{m}^p = \range{C_{\Psi}} \oplus \kernel{D_{\Psi}}$.
\end{lem.}

We will apply the decisive Proposition~\ref{prop:Greoch0}  to the Gram matrix
$G_\Psi $ of a localized frame. For this we need to know that various
weighted Gram matrices also belong to the same matrix algebra, in some sense an extension of \eqref{eq:c7}.

\begin{prop.} \label{def:multlocal2}
Let $\cA$ be an ARI-algebra, and $\Psi$ an $\cA$-localized frame, with $G_\Psi^\mu \in \cA$, additionally. Then the following Gram matrices are also in the matrix algebra $\cA$: 
 $G_\Psi^{1/\mu}$, $G_{\DPsi}^\mu$, $G_{\DPsi}^{1/\mu}$, $G_{\Psi,\DPsi}^\mu$, and $G_{\Psi,\DPsi}^{1/\mu}$.
\end{prop.} 
\begin{proof}
  First we note that the Gram matrix is self-adjoint, therefore $G_\Psi ^{1/\mu
  } =  (G_\Psi ^\mu )^*$ is in $\cA $ (this is where we
  use that $\cA $ is a $^*$-algebra). 
  
Note that if $B$ has closed range in $\ell^2_\mu$, so has $B^\mu$ (in $\ell^2$). Indeed consider $c \in \ell^2$, then 
$$ \norm{\ell^2}{B^\mu c} = \norm{\ell^2}{\diag{\mu} \cdot B
  \cdot \diag{\tfrac{1}{\mu}} c} = \norm{\ell^2_\mu}{B
  \cdot \diag{\tfrac{1}{\mu}} c} \ge $$
  $$ \ge m \cdot  \norm{\ell^2_\mu}{\diag{\tfrac{1}{\mu}}c}  =
   m \cdot  \norm{\ell^2}{c}. $$
   Then it is easy to see that the pseudo-inverse of a matrix $B$ with
  closed range  satisfies  
  ${\left(B^{\mu}\right)}^\dagger =
  {\left(B^\dagger\right)}^\mu$. 
  We know that $\cA$ is pseudo-inverse-closed \cite[Thm.~3.4]{forngroech1}, and therefore if $B^\mu \in \cA$ then also $\left(B^\mu\right)^\dagger$. In particular we have $\left(G_\Psi^\mu\right)^\dagger =  \left(G_\Psi^\dagger\right)^\mu = G_{\DPsi}^\mu \in \cA$.
  Since 
  $G_{\Psi,\DPsi} =
  G_{\Psi}^\dagger  G_{\Psi}$ by 
  \eqref{eq:re7}, we obtain
  that  $G_{\Psi,\DPsi}^\mu =
  \big(G_{\Psi}^\dagger \big)^\mu G_{\Psi}^\mu$ is also in $\cA
  $. 
Since $G_{\Psi,\DPsi}^{1/\mu} = \left( G_{\DPsi,\Psi}^\mu \right)^* =  \left( G_{\Psi,\DPsi}^\mu \right)^*$ we are done.

\end{proof}

\subsection{Galerkin-like Representations of Operators} 
To use the power of the underlying matrix algebra, we next represent
operators by their matrix with respect to a frame~\cite{xxlgro14,xxlframoper1}. 

Let $\Psi$ and $\Phi$ be frames with $\Psi \sim_{\cA} \Psi$ and $\Phi\sim_{\cA} \Psi$.
Let $m_1, m_2 \in W$  
and  $1 \le   p_1,p_2 \le \infty$. 
Then for a  bounded linear operator $O :
  \Hil_{m_1}^{p_1} \rightarrow \Hil_{m_2}^{p_2}$ we define the matrix $
  \Mate^{\left(\Phi , \Psi\right)}(O)$ 
by its entries \begin{equation} \label{eq:Galerkin}
{\Matt{\Phi , \Psi}{O}}_{k,l} = 
\left<O \psi_l, \phi_k
\right>_{\Hil^{p_2}_{m_2},\Hil_{1/m_2}^{p_2'}}
\qquad k,l\in K \, . 
\end{equation}
 In the notation of the frame operators the operator correpsonding to this matrix is
 given as
 $$
 \Matt{\Phi , \Psi}{O} = C_\Phi O D_\Psi \, .
 $$
If $O \in \mathcal{B}(  \Hil_{m_1}^{p_1},
\Hil_{m_2}^{p_2})$, then by the mapping properties of
$C_\Phi $ and $D_\Psi $, the matrix $ \Matt{\Phi , \Psi}{O}$  induces a bounded operator from $\ell_{m_1}^{p_1}$ into $\ell_{m_2}^{p_2}$. 
We call $\Matt{\Phi , \Psi}{O}$
 the \emph{(Galerkin) matrix of $O$ with respect to $\Phi$
   and $\Psi$}.
Conversely, every    operator (matrix) 
  $M\in
  \BL{\ell_{m_1}^{p_1},\ell_{m_2}^{p_2}}$ defines an
  operator 
$ {\Ope}^{(\Phi , \Psi)}(M) \in
 \BL{\Hil_{m_1}^{p_1},\Hil_{m_2}^{p_2}}$ by
\begin{equation} \label{sec:opematr2} \left( \Opp{\Phi , \Psi}{M} \right) f = \sum \limits_k  \Big( \sum
    \limits_j M_{k,j} \left<f, \psi_j\right> \Big) \phi_k  = D_\Phi M C_\Psi f,
\end{equation}
for $f \in \Hil_{m_1}^{p_1}$ with unconditional convergence for $1 \le p_i < \infty$ and unconditional weak$*$-convergence otherwise. 
We call $\Opp{\Phi , \Psi}{M}$ the \emph{operator of
  $M$ with respect to $\Phi$ and
  $\Psi$}. Altogether, we obtain a mapping 
$ {\Ope}^{(\Phi , \Psi)}:
\BL{\ell_{m_1}^{p_1},\ell_{m_2}^{p_2}} \to 
 \BL{\Hil_{m_1}^{p_1},\Hil_{m_2}^{p_2}}   $ from matrices to
operators between coorbit spaces.  

In \cite{xxlgro14} it was shown that $\Ope$ is a left-inverse
of $m$ as follows. 
\begin{prop.} \label{sec:propmatropfram1} Let $\Psi$ and $\Phi$
  be $\cA$-localized frames in $\Hil$ satisfying $\Phi\sim
  _{\mathcal{A}} \Psi $ and
  $m_1, m_2\in W$, 
  and $1 \le p_1,p_2 \le
  \infty$ or $p_1,p_2= 0$. Then 
$$ \left( \Op{\Phi , \Psi}{} \circ \Mat{\DPhi, \DPsi}\right)  = \left( {\Op{\DPhi, \DPsi} \circ \Mat{\Phi , \Psi}
}\right) = \identity{\BL{\Hil_{m_1}^{p_1},\Hil_{m_2}^{p_2}}}.$$
\end{prop.}

Within  this formalism,  a frame  multiplier corresponds to a  diagonal
matrix in the sense that
$$
\MM_{m, \Psi, \Phi} = \Opp{\Phi , \Psi}{\diag{m}}
= D_{\Phi } \diag{m} C_{\Psi } 
\, .
$$

\subsection{Invertibility of operators on coorbit spaces.} The factorization 
\sloppy \mbox{$\Matt{\Phi , \Psi}{O}= C_\Phi O D_{\Psi }$} implies that
$\kernel{D_\Psi}  \subseteq  \kernel{\Matt{\Phi , \Psi}{O}}$,  which is always non-trivial for redundant frames. Therefore
$\Matt{\Phi , \Psi}{O}$ can never be  invertible on the whole
sequence space for redundant
frames, but only on the range of the analysis operator. Using the
projection property of $G_{\Psi, \widetilde \Psi}$ from  Lemma
\ref{lem:gramoper1}, the invertibility of an operator on a coorbit
space can be understood  by means of its
Galerkin matrix as follows. 
\begin{lem.}\label{inverse0} Let $\Psi$ and $\Phi$
  be $\cA$-localized frames in $\Hil$ satisfying $\Phi\sim
  _{\mathcal{A}} \Psi $. Let $m$ be 
  an $\mathcal{A}$-admissible weight and $1 \le p \le
  \infty$ or $p= 0$. 
  A bounded operator $O: \Hil_{m}^p \rightarrow \Hil_{m}^p $ is invertible, if and only if the matrix
  $B_O = \Matt{\Psi, \Phi}{O} + \left( \identity{} - G_{\Psi,\DPsi} \right)$ is invertible on $\ell_{m} ^p$. 
\end{lem.}
The proof of this result can be found in \cite{grpfto22}, Lemma 4.10  and
Thm. 4.11., stated there for Gabor frames and pseudodifferential
operators, but using only general frame properties. 
For an alternative proof using the pseudoinverse $M^\dagger$ of a matrix
 $M$ and the formula $\Matt{\Phi , \Psi}{O\inv}= \Matt{\DPsi ,
   \DPhi}{O} ^\dagger $ see~\cite[Lemma~11]{xxlgro14}.

With this approach we can extend the spectral invariance result of
Corollary~\ref{cor:specinv1}  to the level of  Hilbert space, and we
obtain a universal invertibility result  \cite{xxlgro14}:
 \begin{theorem} \label{sec:theorem6} 
Let $\Psi$ and $\Phi$
  be $\cA$-localized frames in $\Hil$ satisfying $\Phi \sim
  _{\mathcal{A}} \Psi $.
  Assume that $O: \Hil \to \Hil $
  is invertible and that $ \Matt{\Psi  , \DPsi}{O}
  \in \mathcal{A}$. Then $O$ is invertible simultaneously on all
  coorbit spaces $\Hil^p_{m}$, $1\leq p \leq \infty $ for all $\cA$-admissible weights 
  $m$ and $1 \le p \le
  \infty$ or $p= 0$.  
 \end{theorem}

\section{Lifting Properties for Localized Multipliers}
\label{sec:mulisom0}

After setting up the theory of localized frames and matrix algebras,
we can  finally provide a rigorous formulation of our main result and
prove 
the  lifting theorem for coorbit spaces. 

\begin{theorem}[Abstract lifting theorem for coorbit
  spaces] \label{th:multrecip} Let $\cA$ be an ARI-algebra and
  $\mu$, $m \sqrt{\mu}$ and
  $\frac{m}{\sqrt{\mu}}$ be $\cA$-admissible weights.  
Let $\Psi$ be $\cA$-localized frame satisfying the additional
conditions that
$  G_{\Psi}^{\mu} \in \cA$. 

 Then the frame multiplier 
$\MM_{\mu, \Psi} $ is an isomorphism   from  $ \Hil_{m
  \sqrt{\mu}}^p$ onto $  \Hil_{m/\sqrt{\mu}}^p$. 
Likewise, $\MM_{1/\mu , \Psi}$ is an isomorphism from $\Hil_{\frac{m}{\sqrt{\mu}}}^p$ onto $\Hil_{m \sqrt{\mu}}^p$.
\end{theorem}

The remainder of this section is devoted to the proof of the main
theorem, which we  split  into lemmas. 

\subsection{Boundedness of the Multiplier}
We first formulate the mapping properties of frame multipliers on
coorbit spaces. Throughout we assume that $\cA$ is an ARI-algebra.  
For this setting one can go beyond the bounded symbol setting  \cite{xxlmult1} to get bounded operators. 
\begin{lem.} \label{lem:boundmul1}  Let $\Psi$ and $\Phi$ be $\cA$-localized frames.
Let ${\mu}$ and $m$  be  weights such that both  ${m
  \sqrt{\mu}}$ and ${m/\sqrt{\mu}}$ are  $\cA$-admissible. 
Then for all $1\le p \le \infty$ and $p= 0$
$$ \MM_{\mu,\Phi,\Psi} \in \BL{\Hil_{m \sqrt{\mu}}^p,\Hil_{\frac{m}{\sqrt{\mu}}}^p} \text{ and } \MM_{1/\mu,\Phi,\Psi}  \in \BL{\Hil_{\\frac{m}{\sqrt{\mu}}}^p,\Hil_{m \sqrt{\mu}}^p}.
$$
\end{lem.}
\begin{proof} Recall that the frame multiplier factorizes as 
$ \MM_{\mu,\Phi,\Psi} = D_\Phi ~ \diag{\mu} ~ C_\Psi$. Since
$\diag{\mu}$ is an isomorphism from   
$\ell_{m \sqrt{\mu}}^p$ onto
$\ell_{\frac{m}{\sqrt{\mu}}}^p$ by
 Lemma~ \ref{lpiso}, and since both  $D_\Phi$ and
$C_{\Psi}$ are bounded by \eqref{eq:dbound}, we obtain the result. 
\end{proof}

\subsection{Coercivity of the Multiplier}

For the next step we use a coercivity argument that is familiar from the theory of
elliptic partial differential operators~\cite{evans}. In the
context of time-frequency analysis a similar argument was used
in~\cite{grochenig2011isomorphism}. 

\begin{theorem} \label{th:coerc} Let $\cA$ be an ARI-algebra. 
Let $\sqrt{\mu}$ be an $\cA$-admissible weight and $\Psi$ be an $\cA$-localized frame.
 Then the frame multiplier 
$$\MM_{\mu, \Psi} : \Hil_{\sqrt{\mu}}^2 \rightarrow  \Hil_{1/\sqrt{\mu}}^2 $$
is an isomorphism.

Likewise, $\MM_{1/\mu, \Psi}$ is an  isomorphism from $\Hil_{1/\sqrt{\mu}}^2 $ onto $\Hil_{\sqrt{\mu}}^2$.
\end{theorem}
\begin{proof}  Using  Lemma \ref{lem:boundmul1} with $m
  \equiv 1$, the frame multiplier  $\MM_{\mu,\Psi} :
  \Hil_{\sqrt{\mu}}^2 \rightarrow \Hil_{1/\sqrt{\mu}}^2$ is
  bounded. As in the case of elliptic differential operators we
  introduce the  sesquilinear form 
  $$[f,g] \isdef \langle \MM_{\mu, \Psi} f,  g\rangle_{\Hil_{1/\sqrt{\mu}}^2 \times \Hil_{\sqrt{\mu}}^2} $$ for $f,g \in \Hil_{\sqrt{\mu}}^2$. 
By the duality relation \eqref{eq:dualrel},  the dual space of
  $\Hil^2_{\sqrt{\mu}}$   is  $ \Hil^2_{1/\sqrt{\mu}}$, consequently
  both brackets above are 
  well-defined. 
  Therefore 
we obtain 
$$\left[ f , f \right] = \left< \MM_{\mu, \Psi} f, f\right>_{\Hil_{1/\sqrt{\mu}}^2 \times
  \Hil_{\sqrt{\mu}}^2} =  \left< \diag{\mu} C_{\Psi} f, C_{\Psi} f\right>_{\ell_{1/\sqrt{\mu}}^2,\ell_{\sqrt{\mu}}^2} $$
$$ = \sum \limits_k \mu_k \left(C_{\Psi} f\right)_k
\overline{\left(C_{\Psi} f\right)_k} = \norm{\ell_{\sqrt{\mu}}^2}{
  C_{\Psi} f }^2 \, .$$ 
Note that the $\Hil ^2_{{\sqrt{\mu}}}\,$-norm of $f$ is defined by
$\|C_\DPsi f\|_{\ell ^2_{\sqrt{\mu}}}$ with respect to the dual
frame. By the main theorem of localized frames
(Proposition~\ref{locmain2}) this norm is equivalent to $\|C_\Psi f
\| _{\ell ^2_{\sqrt{\mu}}}$. Thus we conclude that
$$
[f,f] = \|C_\Psi f \| _{\ell ^2_{\sqrt{\mu}}}^2  
\stackrel{\eqref{eq:eqnorm1}}{\asymp} \|C_\DPsi f\|_{\ell ^2_{\sqrt{\mu}}} = \|f\|^2_{\Hil_{\sqrt{\mu}}^2}
\, .$$
Consequently, if $\MM_{\mu, \Psi} f = 0$ in $\Hil_{\sqrt{\mu}}^2$, then $C_{\Psi} f = 0$, and therefore $f=0$. So $\MM_{\mu, \Psi}$ is injective.

For a converse inequality, we observe that for $f \not=0$
\begin{align*}
   \|\MM_{\mu, \Psi} f \|_{\Hil_{\frac{1}{\sqrt{\mu}}}^2} &= \sup _{\|g\|_{\Hil_{\sqrt{\mu}}^2} =1}
   \big| \langle \MM_{\mu, \Psi} f,  g\rangle_{\Hil_{1/\sqrt{\mu}}^2 \times \Hil_{\sqrt{\mu}}^2}  \big|  \\
 &  \geq \ \big| \Big\langle \MM_{\mu, \Psi} f,  \frac{f}{\| f
     \|_{\Hil_{\sqrt{\mu}}^2}}\Big\rangle_{\Hil_{1/\sqrt{\mu}}^2
     \times \Hil_{\sqrt{\mu}}^2} \big| = \frac{[f,f]}{\| f
     \|_{\Hil_{\sqrt{\mu}}^2}}\asymp 
   \|f\|_{\Hil_{\sqrt{\mu}}^2} \, .
  \end{align*}
     Consequently, $\MM_{\mu,\Psi} : \Hil_{\sqrt{\mu}}^2 \rightarrow \Hil_{1/\sqrt{\mu}}^2$ has
closed  range. By the closed range theorem, its adjoint operator
   $\MM_{\mu,\Psi}^* = \MM_{\mu,\Psi}$ is therefore onto. It
   follows that   $\MM_{\mu,\Psi} $ is an isomorphism between $
   \Hil_{\sqrt{\mu}}^2$and $ \Hil_{1/\sqrt{\mu}}^2$, as  claimed. 
\end{proof}

\subsubsection{The Galerkin Matrix of $\MM_{1/\mu, \Psi} \MM_{\mu, \Psi}$}
As an easy consequence of Theorem \ref{th:coerc}, the  composition of
these two frame multipliers gives an isomorphism of the weighted
Banach spaces into itself.  
\begin{cor.} \label{comp1}  Let 
$\Psi$ be a $\cA$-localized frame and  
 $\sqrt{\mu}$ be an $\cA$-admissible weight.  

(i) The composition $\MM_{1/\mu, \Psi} \MM_{\mu, \Psi}$ is an
  isomorphism on $\Hil_{\sqrt{\mu}}^2$. Likewise,  
   $\MM_{\mu, \Psi} \MM_{1/\mu, \Psi} $ is an isomorphism  on $\Hil_{\frac{1}{\sqrt{\mu}}}^2$. 

   (ii) The matrix associated to $\MM_{\frac{1}{\mu}, \Psi}
   \MM_{\mu, \Psi}$ is given by 
\begin{eqnarray} \label{eq:re5}
   \Matt{\Psi,\Psi}{\MM_{\frac{1}{\mu}, \Psi} \MM_{\mu,
  \Psi}} =  G_{\Psi} G_{\Psi}^{1/\mu} G_{\Psi}. 
\end{eqnarray}
 \end{cor.}

 \begin{proof}
   (i) follows from the mapping properties of
   Lemma~\ref{lem:boundmul1} and from  Theorem~\ref{th:coerc}. 

(ii) The formula of its matrix representation is given by
\begin{eqnarray} \label{eq:matrmul2}
   \Matt{\Psi,\Psi}{\MM_{\frac{1}{\mu}, \Psi} \MM_{\mu, \Psi}} \nonumber
   &=&  C_\Psi \MM_{1/\mu, \Psi} \MM_{\mu, \Psi} D_\Psi \nonumber \\
   &=& C_\Psi D_\Psi \diag{1/\mu} C_\Psi D_\Psi \diag{\mu} C_\Psi D_\Psi \nonumber \\
   &=& G_{\Psi} G_{\Psi}^{1/\mu} G_{\Psi}. 
\end{eqnarray}
   
 \end{proof}

 \subsection{Proof of the General Lifting Theorem}
 We finally prove the main theorem and show that the frame multiplier 
$\MM_{\mu, \Psi} $ is a (Banach space) isomorphism from $
\Hil_{m \sqrt{\mu}}^p$ onto $
\Hil_{m/\sqrt{\mu}}^p$ whenever all occurring weights are
$\cA$-admissible.  
\begin{proof}[Proof of Theorem~\ref{th:multrecip}]
(i)   We  consider the matrix
	\begin{equation} \label{eq:split1}
  B = \Matt{ \Psi,\Psi}{\MM_{1/\mu,
    \Psi} \MM_{\mu, \Psi}} + \left( \identity{} -
  G_{\Psi,\DPsi}\right)\, .
\end{equation}
Since $\MM_{1/\mu,
    \Psi} \MM_{\mu, \Psi}$ is invertible on $\Hil ^2 _{\mu }$ by
  Corollary~\ref{comp1},  we obtain from Lemma \ref{inverse0}  that
  $B$  is an isomorphism on $\ell_{\sqrt{\mu}}^2$.
  
(ii) Next we check that $B\in \cA $. Since by assumption  $G_{\Psi
}$ and $G_{\Psi }^\mu$ are in  $\cA $,
Proposition~\ref{def:multlocal2} implies that also  $G_{\Psi }^{1/\mu }\in
\cA $ and $G_{\Psi, \DPsi}  \in \cA $. The
factorization~\eqref{eq:re5} now   
implies  that  $B \in \cA$. 

(iii) We need to  show that also $B^\mu \in \cA$.
A small computation shows that 
\begin{eqnarray*} B^\mu  & = &  \diag{\mu} \cdot B \cdot \diag{1/\mu} \\
& = & \diag{\mu} \left( G_{\Psi} G_{\Psi}^{1/\mu} G_{\Psi}  + \left( \identity{} - G_{\Psi,\DPsi}\right) \right)  \diag{1/\mu}  \\
& = &  G_{\Psi}^{\mu} G_{\Psi}
      G_{\Psi}^{\mu}  +\identity{} -
      G_{\Psi,\DPsi}^{\mu}\, .
\end{eqnarray*}
From Proposition \ref{def:multlocal2} it follows that  $B^\mu \in \cA $ as well.

(iv) Since $B$ is invertible on $\ell ^2_\mu$ and
$B^\mu \in \cA $,  Proposition \ref{prop:Greoch0} implies
that  $B$ is (boundedly)  invertible  on all $\ell^p_{\mu
  w}$ for $1\le p\le\infty$ and $w \in W$.  
By assumption this includes the space $\ell^p_{m\sqrt{\mu}}$.
Therefore, by Lemma \ref{inverse0} again,  $\MM_{1/\mu, \Psi} \MM_{\mu, \Psi}$ is invertible on $\Hil_{m \sqrt{\mu}}^p$. 
Consequently, $\MM_{\mu, \Psi}$ is one-to-one  on $\Hil_{m\sqrt{\mu}}^p$.

(v) We now replace $\mu $ by $1/\mu $ and repeat  the argument for
$\MM_{\mu, \Psi} \MM_{1/\mu, \Psi}$. We conclude the invertibility
of $\MM_{\mu, \Psi} \MM_{1/\mu, \Psi}$ and consequently  $\MM_{\mu, \Psi}$ is surjective  onto
$\Hil_{\frac{m}{\sqrt{\mu}}}^p$,  and the isomorphism
property of $\MM_{\mu, \Psi}$  is shown. 

The preceding arguments show that  $\MM_{1/\mu, \Psi} $ is
one-to-one and onto and thus an isomorphism as well. 
\end{proof}

\begin{rem.} \label{rem:cont1}
{\rm A version of Theorem~\ref{th:multrecip} can also be shown for  continuous frames
\cite{antoin2} and corresponding multipliers \cite{xxlbayasg11}. The
proof structure remains the same, the use of integral operators
instead of matrices requires some technical modifications.} 
\end{rem.}

\section{Gabor Multipliers and Toeplitz Operators} \label{sec:appl0}

In this section we 
make explicit what the abstract lifting  theorems says
for Gabor frames and Toeplitz operators on Fock space. In both
examples we look at a  class of structured frames, then  explain
the meaning of localization, describe the associated coorbit space, and finally formulate a concrete version of the lifting
theorem. By choosing particular inverse-closed matrix algebras from
the literature, one can formulate many new and interesting lifting
theorems for Gabor multipliers and Toeplitz operators. We leave this
exercise to the reader, as they are already contained in Theorem~\ref{th:multrecip}. 

\subsection{Gabor Multipliers} \label{sec:Gabmult1}
We consider Gabor frames since these  were one of the  motivations for the
introduction of  localized frames. Gabor frames are structured frames
that are generated by \tfs s of a single function.
As usual, 
 for $x, \omega ,t \in \rd$ and $\lambda = (x,\omega )\in \rdd $, 
the \emph{time-frequency shift} $\pi \left( \lambda \right)$ is
defined by $\pi (\lambda )f(t)  =  e^{2 \pi i
   \omega t} f(t-x)$.  
      
 A \emph{Gabor frame}  is a frame consisting of \tfs s   $\Psi = \left\{ \pi
   (\lambda) g : \lambda \in \Lambda\right\}$. Often $\Lambda $ is
 taken to be a lattice in $\rdd $, but for the lifting theorems we
 only need to assume that
 $\Lambda \subseteq \rdd $ is relatively separated.  

The frame multipliers corresponding to a Gabor frame are called  Gabor
multipliers and have already been studied in the context of \tfa
\cite{feinow1,benepfand07,babacofesc22}. Given a symbol $\mu$ on $\Lambda $ and a window
$g$, the
associated  multiplier is defined as
$$
m_\mu f = \sum \limits_{\lambda \in \Lambda} \mu_\lambda \, 
\ip{f}{\pi(\lambda) g}\, \pi(\lambda) g \, .
$$ 

To describe the localization of a Gabor frame operator we use the
Jaffard algebra of Example~\ref{exjaffard} and impose a condition of
polynomial off-diagonal decay on the Gram matrix.  Using the \stft ,
one can find condition such that the Gram matrix  $G$ belongs to the Jaffard
algebra $\cA _s$, i.e. for all $ \lambda  , \lambda ' \in \Lambda $
\begin{align} \label{eq:rev1}
  |G_{\lambda ,\lambda ' } | = |\langle \pi (\lambda ')g, \pi (\lambda
  )g \rangle |
  = | \langle g, \pi (\lambda -\lambda ') \rangle|  \leq C
  (1+|\lambda -\lambda ' |)^{-s} \,
  \, .
\end{align}
We have  already seen that the class of $t$-moderate weights, i.e.,
$m$ satisfies $m(\lambda +\lambda ' ) \leq C (1+|\lambda | )^t
m(\lambda ' )$ for all $\lambda , \lambda '\in \Lambda $, is $\cA _s$
admissible for $t<s-2d $ (note that $\Lambda \subseteq \rdd$).

The coorbit spaces associated to a Gabor frame can be identified with
the class of \modsp s and possess a more explicit
description. To describe these spaces,  we define the \emph{short-time Fourier
  transformation} of a function $f$ with respect to a non-zero   \emph{window}
function  $g\in \LtRd $ to be   
$$ V_g f(x,\omega) = \left< f, \pi(\lambda) g \right>_{\LtRd} = \int \limits_{\rd} f(t) 
      \overline{g(t - x)} e^{-2\pi i \omega t} \, dt \,.$$
Then the \modsp s are defined as follows. 

\begin{def.} Fix a non-zero window $g \in \cS$, a $t$-moderate weight
  function $m$ on $\rdd$ for some $t>0$ and $1\le p\le \infty$. Then the \emph{modulation space}  $\Mpm$  consists of all tempered distributions $f \in \mathcal S' (\rd)$ such that $V_g f \in \Lpm$. The norm on $\Mpm$ is
$$ \norm{\Mpm}{f} = \norm{\Lpm}{V_g f} $$
\end{def.}
It can be shown \cite{gr01} that the spaces defined above do not depend on the
special choice of the non-zero test function $g$, as long as it is
sufficiently well concentrated in the time frequency sense. Different
functions define the same space with equivalent norms. Moreover these
functions spaces are Banach spaces that are invariant under time-frequency
shifts  \cite{gr01}.
 Note that the Gram matrix $G $ of a Gabor frame is in the Jaffard algebra $ \cA
_s$ - see \eqref{eq:rev1} -  if and only if the window $g$ is in the modulation space $
M^\infty _{w_s}$ for the weight function $w_s(z)  = (1+|z|)^s$. Thus
in \tfa\  properties of the Gram matrix can be translated directly
into properties of the underlying window $g$. 

As a first application of localized frames  in~\cite{forngroech1},  the coorbit spaces with respect
to a Gabor frame $\{ \pi (\lambda )g: \lambda \in \Lambda \}$
was shown to be identical to a \modsp . Precisely,  $\Mpm =
\Hil_{\tilde{m}}^p$, where $\tilde{m}$ is the restriction of the
weight $m$ to $\Lambda $. 
So our lifting theorem leads to the following result in the context of
\tfa . 
\begin{theorem} \label{th:liftGab} 
Fix $s >2d$ and let $a>2s-2d$. Let $g \in M_{w_a}^{\infty}
(\rd)$ and  assume that  $\Lambda \subseteq \rd$ be a relatively separated set, such that $\left\{ \pi(\lambda) g \right\}$ is a Gabor frame. 
Assume that $\mu , m\sqrt{\mu }$, and $m/\sqrt{\mu }$ are
$t$-moderate weights for $t<s-2d$. 

Then,  for all $1\le p \le \infty$, the  Gabor multiplier
 $m_\mu$ is an isomorphism from $M_{m \sqrt{\mu}}^p(\rd)$
 onto $M_{\frac{m}{\sqrt{\mu}}}^p(\rd)$. 
\end{theorem}

\begin{proof}
The underlying ARI-algebra is the Jaffard algebra $\cA _s$ discussed
in Example~\ref{exjaffard}.  As the  class of $\cA _s$-admissible
weights we take 
 all $t$-moderate weights with $t<s-2d$, whence comes the
assumption on the weights.

If $g\in M_{w_a}^{\infty} (\rd)$, then the Gram matrix of the frame
satisfies \cite{forngroech1}
$$
  |G_{\lambda ,\lambda ' } | 
  = |V_gg(\lambda -\lambda ' )| \leq C
  (1+|\lambda -\lambda ' |)^{-a} \leq C (1+|\lambda -\lambda ' |)^{-s}  \, \quad \forall \lambda  , \lambda ' \in \Lambda
  \, , 
$$
and consequently $G\in \cA _s$. In view of
Theorem~\ref{th:multrecip} we also need that $ G^{\mu} =
\diag{\mu}  \, \, G\, \, 
 \diag{\tfrac{1}{\mu}}   \,  
    \in \cA _s$. Note that $G^\mu $ has the entries $\frac{\mu
    _\lambda }{ \mu _{\lambda '}} G_{\lambda ,\lambda ' }$
    and that $\mu _\lambda \leq (1+|\lambda - \lambda '|)^t\,  \mu
    _{\lambda '}$ because $\mu $ is $t$-moderate. Therefore
    $$
    \frac{\mu     _\lambda }{\mu _{\lambda '}} \,  |G_{\lambda ,\lambda '
    }|\leq (1+|\lambda - \lambda '|)^{t-a} \qquad \text{ for all }
    \lambda , \lambda ' \in \Lambda \, .
    $$
    Since $a>2s-2d>s + t$  and thus $t-a<-s$ by assumption, we
    conclude that  $G^{\mu} \in \cA
    _s$ and $G^{1/\mu}\in \cA _s$. 

We have verified all hypothesis of Theorem~\ref{th:multrecip} and
conclude that the Gabor multiplier   $m_\mu $ maps  the
space $\Hil ^p_{m\sqrt{\mu } } = M^p_{m\sqrt{\mu } } (\rdd ) $ bijectively
to $\Hil ^p_{m/\sqrt{\mu } }= M^p_{m/\sqrt{\mu }
}$. 
\end{proof}

\textbf{Discussion of the literature.} It is not difficult to see and
was already proved by Feichtinger~\cite{fei83}  that for pure frequency  weights 
$m(x,\omega ) = (1+|\omega |)^s$,  the lifting operator is the
differentiation operator. 
A new generation of lifting theorems covers the class of general
    \modsp s $\Mmpq $ with a weight function $m$ depending on both
    time and frequency. 
    In this case, different
    \modsp s are related by \emph{\tf\ multipliers}, or, in a different
    terminology, by Toeplitz operators with respect to the \stft
    . These  operators are parametrized by a
    symbol $m$ depending on time \emph{and} frequency and are   much  more
    complicated  than Fourier  multipliers. Formally, they are defined
    by the vector-valued integral 
    \begin{equation}
      \label{eq:c9}
    A_\mu^g f = \intrdd \mu(z) V_gf(z) \, \pi (z) g \, dz  \, .  
    \end{equation}
    Generically,   the lifting theorems for
    general \modsp s assert that the \stft\ multiplier $A^g_\mu$ is an
    isomorphism from $M^{p,q}_{m \sqrt{\mu }} $ onto
      $M^{p,q}_{m/\sqrt{\mu }}$. The 
    proofs require a heavy dose of
    \tfa , \psdo\ theory, and abstract Banach algebra
    arguments. Various levels of generality of this lifting theorem
    are recorded ranging from  weights of
    polynomial growth to weights (sub)exponential growth ~\cite{ACT20,BToft05,grto10,grochenig2011isomorphism}% , for

    Theorem~\ref{th:liftGab} adds a further facet to these lifting
    theorems in that the \emph{function} $\mu $ in \eqref{eq:c9} is replaced by a Radon
    measure $\mu = \sum _{\lambda \in \Lambda } \mu _\lambda \delta
    _\lambda $. Theorem~\ref{th:liftGab} could possibly be proved by
    adding more technicalities to the methods
    in~\cite{grto10,grochenig2011isomorphism}, whereas 
    in our set-up the result is  a direct consequence of the abstract
    lifting theorem with localized frames. Extending the result in the current paper to the continuous setting as mentioned in Rem \ref{rem:cont1} would allow to directly show the lifting theorem for the Toeplitz operator $A_\mu^g$ with the techniques used in here.

\subsection{Fock Spaces} 
Fock spaces are Banach spaces of entire functions satisfying certain
integrability conditions. Given a non-negative weight function $m$ on
$\bC $, the Fock space $\cF ^p_m$ consists of all entire functions
with finite  norm
$$
\|F\|_{\cF ^p_m} = \Big( \int _{\bC } |F(z)|^p m(z)^p e^{-\pi p
  |z|^2/2} \, dz \Big) ^{1/p} \, .
$$
For $p=2$ and $m\equiv 1$, $\cF ^2$ is a reproducing kernel Hilbert
space  with the normalized reproducing kernel \cite{folland89,zh12} 
$$\psi _\lambda (z) = e^{\pi \bar{\lambda } z} e^{-\pi |z|^2/2} \, ,$$
so that $f(\lambda) ^{-\pi |\lambda |^2/2}  = \langle f, \psi _\lambda \rangle _{\cF ^2}$. 
The frames of reproducing kernels in $\cF ^2$ are completely
characterized by their density. The Theorem of
Lyubarskii-Seip-Wallst\'en~\cite{lyub92,sewa92} asserts that \emph{\{$ \psi _\lambda
  :\lambda \in \Lambda \}$ is a frame for $\cF ^2$, \fif\  $\Lambda
  $ contains a relatively separated subset $\Lambda '\subseteq \Lambda
  $ with  lower
Beurling $D^-(\Lambda ' ) = \lim _{r\to \infty } \inf _{z\in \bC } \frac{\mathrm{card}\,
  \Lambda \cap B_r(z)}{\pi r^2} >1$}. 
Since
$$
|\langle \psi _{\lambda '}, \psi _\lambda  \rangle |= |e^{\pi
  \bar{\lambda '} \lambda } |  e^{-\pi (|\lambda |^2 + |\lambda '|^2)/2}
=  e^{-\pi |\lambda - \lambda'|^2/2}
$$
has Gaussian decay, the Gramian $G$ of a frame of reproducing
kernels $\{\psi _\lambda \}$ is contained in every known matrix
algebra defined by off-diagonal decay. We can therefore use arbitrary
moderate weights of subexponential decay, i.e., for some $C,\alpha >0$
and $0<\beta <1$, we have
\begin{equation}
  \label{eq:c11}
  m(\lambda +\lambda ') \leq C e^{\alpha
  |z|^\beta } m(\lambda ') \qquad \text{ for all }\lambda , \lambda ' \in \bC \, .\end{equation}

In this context a frame multiplier with multiplier sequence $\mu =
(\mu _\lambda )_{\lambda \in \Lambda }$  takes the form 

$$ M_\mu F (z) = \sum \limits_{\lambda \in \Lambda} \mu_\lambda \,
F(\lambda) e^{\pi \bar{\lambda } z} \,  e^{-\pi
  \left|\lambda\right|^2}\, . $$
As in the case of time-frequency multipliers we note that $M_\mu $ can
be read as a Toeplitz operator on Fock space whose symbol is a Radon
measure in place of a nice function.

For this example the abstract lifting theorem
(Theorem~\ref{th:multrecip}) takes the following incarnation. 

\begin{theorem}
  Assume that $\mu $ is moderate as in \eqref{eq:c11} and that
  $\Lambda \subseteq \bC $ is relatively separated such that 
  $D^-(\Lambda )>1$. Then then $M_\mu$ is an isomorphism from
  $\cF^p_{m \sqrt{\mu }}$ to $\cF^p_{\frac{m}{\sqrt{\mu}}}$.
\end{theorem}

This result seems to be new in
complex analysis. For multivariate versions one must assume that the
reproducing kernels $\{ \psi _\lambda \}$ form a frame, but there is
no characterization via the density.

\section*{Acknowledgment}
The authors thank Dominik Bayer, Gilles Chardon, Jose Luis Romero and Lukas K\"ohldorfer for interesting discussions.


\begin{thebibliography}{99}

  \bibitem{ACT20}
A.~Abdeljawad, S.~Coriasco, and J.~Toft.
\newblock Liftings for ultra-modulation spaces, and one-parameter groups of
  {G}evrey-type pseudo-differential operators.
\newblock {\em Anal. Appl. (Singap.)}, 18(4):523--583, 2020.

\bibitem{aliant1}
S.~T. Ali, J.-P. Antoine, and J.-P. Gazeau.
\newblock {\em Coherent States, Wavelets and Their Generalization}.
\newblock Graduate Texts in Contemporary Physics. Springer New York, 2000.

\bibitem{antoin2}
S.~T. Ali, J.-P. Antoine, and J.-P. Gazeau.
\newblock Continuous frames in {H}ilbert space.
\newblock {\em Ann. Physics}, 222(1):1--37, 1993.

\bibitem{Balan2011}
R.~Balan, P.~Casazza, and Z.~Landau.
\newblock Redundancy for localized frames.
\newblock {\em Israel J. Math.}, 185(1):445, 2011.

\bibitem{bacahela06}
R.~M. Balan, P.~G. Casazza, C.~Heil, and Z.~Landau.
\newblock Density, overcompleteness, and localization of frames {I}: Theory.
\newblock {\em J. Fourier Anal. Appl.}, 12(2):105--143, 2006.

\bibitem{xxlmult1}
P.~Balazs.
\newblock Basic Definition and Properties of Bessel Multipliers.
\newblock {\em J. Math. Anal. Appl.}, 325:571-585, 2007

\bibitem{xxlframoper1}
P.~Balazs.
\newblock Matrix-representation of operators using frames.
\newblock {\em Sampl. Theory Signal Image Process.}, 7(1):39--54, 2008.

\bibitem{babacofesc22}
P.~{B}alazs, F.~{B}astianoni, E.~{C}ordero, H.~G. {F}eichtinger, and
  N.~{S}chweighofer.
\newblock {C}omparison between {F}ourier multipliers and {S}{T}{F}{T}
  multipliers: the smoothing effect of the short-time {F}ourier transform,.
\newblock {\em J. Math. Anal. Appl.} 529 (2024), no. 1, Paper No. 127579, 32 pp. 

\bibitem{xxlbayasg11}
P.~Balazs, D.~Bayer, and A.~Rahimi.
\newblock Multipliers for continuous frames in {H}ilbert spaces.
\newblock {\em J. Phys. A: Math. Theor.}, 45:244023, 2012.

\bibitem{xxlgro14}
P.~Balazs and K.~Gr{\"o}chenig.
\newblock A guide to localized frames and applications to {G}alerkin-like representations of operators.
\newblock In I.~Pesenson et al., editors, {\em Frames and Other Bases in Abstract and Function Spaces}, ANHA series. Birkh\"auser/Springer, 2017.

\bibitem{xxllabmask1}
P.~Balazs, B.~Laback, G.~Eckel, and W.~A. Deutsch.
\newblock Time-frequency sparsity by removing perceptually irrelevant components using a simple model of simultaneous masking.
\newblock {\em IEEE Trans. Audio Speech Lang. Process.}, 18(1):34--49, 2010.

\bibitem{xxlstoeant11}
P.~Balazs, D.~Stoeva, and J.-P. Antoine.
\newblock Classification of general sequences by frame-related operators.
\newblock {\em Sampl. Theory Signal Image Process.}, 10(2):151--170, 2011.

\bibitem{Bas90}
A.~G. Baskakov.
\newblock Wiener's theorem and asymptotic estimates for elements of inverse
  matrices.
\newblock {\em Funktsional. Anal. i Prilozhen.}, 24(3):64--65, 1990.

\bibitem{benepfand07}
J.~Benedetto and G.~Pfander.
\newblock Frame expansions for {G}abor multipliers.
\newblock {\em Appl. Comput. Harmon. A.}, 20(1):26--40,  2006.

\bibitem{BToft05}
P.~Boggiatto and J.~Toft.
\newblock Embeddings and compactness for generalized {S}obolev-{S}hubin spaces
  and modulation spaces.
\newblock {\em Appl. Anal.}, 84(3):269--282, 2005.

\bibitem{Boelcskei01noisereduction}
H.~B{\"o}lcskei and F.~Hlawatsch.
\newblock Noise reduction in oversampled filter banks using predictive quantization.
\newblock {\em IEEE Trans. Inf. Theory}, 47:155--172, 2001.

\bibitem{ole1n}
O.~Christensen.
\newblock {\em {A}n {I}ntroduction to {F}rames and {R}iesz {B}ases}.
\newblock Birkh{\"a}user, 2016.

\bibitem{olepinv}
O,~Christensen.
\newblock Frames and pseudo-inverses.
\newblock {J. Math. Anal. Appl}, 195, 401-414, 1995.

\bibitem{CG03}
E.~Cordero and K.~Gr{\"o}chenig.
\newblock Time-frequency analysis of localization operators.
\newblock {\em J. Funct. Anal.}, 205(1):107--131, 2003.

\bibitem{conw1}
J.~B. Conway.
\newblock {\em A Course in Functional Analysis}.
\newblock Graduate Texts in Mathematics. Springer, 2nd edition, 1990.

\bibitem{Dahlkeetal10}
S.~Dahlke, M.~Fornasier, and K.~Gr{\"o}chenig.
\newblock Optimal adaptive computations in the {J}affard algebra and localized frames.
\newblock {\em J. Approx. Theory}, 162:153--185, 2010.

\bibitem{DepKronTor07}
P.~Depalle, R.~Kronland-Martinet, and B.~Torr{\'e}sani.
\newblock Time-frequency multipliers for sound synthesis.
\newblock In {\em Proc. Wavelet XII Conf., SPIE Annu. Symp., San Diego}, 2007.

\bibitem{duffschaef1}
R.~J. Duffin and A.~C. Schaeffer.
\newblock A class of nonharmonic {F}ourier series.
\newblock {\em Trans. Amer. Math. Soc.}, 72:341--366, 1952.

\bibitem{evans}
L.~C. Evans.
\newblock {\em Partial differential equations}, volume~19 of {\em Graduate
  Studies in Mathematics}.
\newblock American Mathematical Society, Providence, RI, second edition, 2010.

\bibitem{fei83}
H.~G. Feichtinger.
\newblock Modulation spaces on locally compact abelian groups.
\newblock In {\em Proceedings of ``International Conference on Wavelets and
  Applications" 2002}, pages 99--140, Chennai, India, 2003.
\newblock Updated version of a technical report, University of Vienna, 1983.

\bibitem{fe06}
H.~G. Feichtinger.
\newblock Modulation spaces: Looking back and ahead.
\newblock {\em Sampl. Theory Signal Image Process.}, 5(2):109--140, 2006.

\bibitem{fg89jfa}
H.~G. Feichtinger and K.~Gr{\"o}chenig.
\newblock Banach spaces related to integrable group representations and their
  atomic decompositions. {I}.
\newblock {\em J. Functional Anal.}, 86(2):307--340, 1989.

\bibitem{feinow1}
H.~G. Feichtinger and K.~Nowak.
\newblock A first survey of {G}abor multipliers.
\newblock In {\em Advances in Gabor Analysis}, pages 99--128. Birkh\"auser, 2003.


\bibitem{folland89}
G.~B. Folland.
\newblock {\em Harmonic Analysis in Phase Space}.
\newblock Princeton Univ. Press, Princeton, NJ, 1989.


\bibitem{forngroech1}
M.~Fornasier and K.~Gr{\"o}chenig.
\newblock Intrinsic localization of frames.
\newblock {\em Constr. Approx.}, 22(3):395--415, 2005.

\bibitem{gr01}
K.~Gr{\"o}chenig.
\newblock {\em Foundations of Time-Frequency Analysis}.
\newblock Birkh\"auser Boston, 2001.

\bibitem{gr04-1}
K.~Gr{\"o}chenig.
\newblock Localization of frames, {B}anach frames, and the invertibility of the frame operator.
\newblock {\em J. Fourier Anal. Appl.}, 10(2):105--132, 2004.

\bibitem{GK10}
K.~Gr{\"o}chenig and A.~Klotz.
\newblock Noncommutative approximation: inverse-closed subalgebras and
  off-diagonal decay of matrices.
\newblock {\em Constr. Approx.}, 32(3):429--466, 2010.

\bibitem{grle06}
K.~{G}r{\"o}chenig and M.~{L}einert.
\newblock {S}ymmetry and inverse-closedness of matrix algebras and symbolic
  calculus for infinite matrices.
\newblock {\em {T}rans. {A}mer. {M}ath. {S}oc.}, 358:2695--2711, 2006.

\bibitem{grpfto22}
K.~Gr{\"o}chenig, C.~Pfeuffer, and J.~Toft.
\newblock Spectral invariance of quasi-{B}anach algebras for matrices and pseudodifferential operators.
\newblock {\em Forum Math.} 36(5):1201--1224, 2024.

\bibitem{grto10}
K.~Gr{\"o}chenig and J.~Toft.
\newblock The range of localization operators and lifting theorems for modulation and {B}argmann-{F}ock spaces.
\newblock {\em Trans. Amer. Math. Soc.}, 365(8):4475--4496, 2010.

\bibitem{grochenig2011isomorphism}
K.~Gr{\"o}chenig and J.~Toft.
\newblock Isomorphism properties of {T}oeplitz operators and pseudo-differential
  operators between modulation spaces.
\newblock {\em J. Anal. Math.}, 114(1):255--283, 2011.

\bibitem{jaffard90}
S.~Jaffard.
\newblock Propri\'et\'es des matrices ``bien localis\'ees'' pr\`es de leur
  diagonale et quelques applications.
\newblock {\em Ann. Inst. H. Poincar\'e Anal. Non Lin\'eaire}, 7(5):461--476,
  1990.

\bibitem{Krishtal2011}
I.~A. Krishtal.
\newblock Wiener's lemma and memory localization.
\newblock {\em J. Fourier Anal. Appl.}, 17(4):674--690, 2011.

\bibitem{Koehldorfer2025}
L.~K{\"o}hldorfer and P.~Balazs.
\newblock Wiener pairs of Banach algebras of operator-valued matrices.
\newblock {\em J. Math. Anal. Appl.} 549 (2025), no. 2, Paper No. 129525, 26 pp. 


\bibitem{lyub92}
Y.~I. Lyubarski{\u\i}.
\newblock Frames in the {B}argmann space of entire functions.
\newblock In {\em Entire and subharmonic functions}, pages 167--180. Amer.
  Math. Soc., Providence, RI, 1992.

\bibitem{marcoxxl11}
M.~Liuni, P.~Balazs, and A.~R{\"o}bel.
\newblock Sound analysis and synthesis adaptive in time and two frequency bands.
\newblock In {\em Proc. 14th Int. Conf. Digital Audio Effects (DAFx-11)}, Paris, 2011.

\bibitem{hlawatgabfilt1}
G.~Matz and F.~Hlawatsch.
\newblock Linear time-frequency filters: On-line algorithms and applications.
\newblock In A.~Papandreou-Suppappola, editor, {\em Applications in Time-Frequency Signal Processing}, pages 205--271. CRC Press, 2002.

\bibitem{kat68}
Y.~Katznelson.
\newblock {\em An Introduction to Harmonic Analysis}.
\newblock John Wiley \& Sons, New York, 1968.

\bibitem{olivtor10}
A.~Olivero, B.~Torr{\'e}sani, and R.~Kronland-Martinet.
\newblock A new method for {G}abor multipliers estimation: Application to sound morphing.
\newblock In {\em EUSIPCO-2010}, 2010.

\bibitem{rahxxlXX}
A.~Rahimi, and P.~Balazs.
\newblock Multipliers for p-Bessel sequences in Banach spaces 
\newblock {\em Integr. Equat. Oper. Th.}, 68, 193-205 , 2010.

\bibitem{balsto09new}
D.~T. Stoeva and P.~Balazs.
\newblock Invertibility of multipliers.
\newblock {\em Appl. Comput. Harmon. Anal.}, 33(2):292--299, 2012.

\bibitem{uncconv2011}
D.~T. Stoeva and P.~Balazs.
\newblock Canonical forms of unconditionally convergent multipliers.
\newblock {\em J. Math. Anal. Appl.}, 399:252--259, 2013.

\bibitem{sewa92}
K.~{S}eip and R.~{W}allst{\'e}n.
\newblock {D}ensity theorems for sampling and interpolation in the
  {B}argmann-{F}ock space. {I}{I}.
\newblock {\em J. Reine Angew. Math.}, 429:107--113, 1992.

\bibitem{sun07a}
Q.~Sun.
\newblock Wiener's lemma for infinite matrices.
\newblock {\em Trans. Amer. Math. Soc.}, 359(7):3099--3123 (electronic), 2007.

\bibitem{triebel83}
H.~Triebel.
\newblock {\em Theory of Function Spaces}.
\newblock Birkh\"auser Verlag, Basel, 1983.

\bibitem{wi64}
N.~Wiener.
\newblock Tauberian theorems.
\newblock {\em MIT Press}, pages 143--242, 1964.

\bibitem{zh12}
K.~{Z}hu.
\newblock {\em {A}nalysis on {F}ock {S}paces}.
\newblock {G}raduate {T}exts in {M}athematics 263. {S}pringer, 2012.

\end{thebibliography}
\end{document}